\documentclass[a4paper, 12pt, twoside]{amsart}

\usepackage{graphicx}
\usepackage{mathtools, yhmath, forloop}
\usepackage[all]{xy}
\usepackage{multirow} 
\usepackage{hyperref}
\usepackage{tikz}
\usepackage{amsmath,amsfonts,amssymb,amsthm, amscd}
\usepackage{enumitem} 
\usepackage{thmtools} 
\usepackage{latexsym}
\usepackage[utf8]{inputenc}
\usepackage[english]{babel}

 \divide\oddsidemargin by 2
\makeatletter
\@namedef{subjclassname@2020}{%
  \textup{2020} Mathematics Subject Classification}
\makeatother

\newcommand{\set}[2]{\lbrace #1 \mid #2 \rbrace}
\newcommand{\PC}[1]{\mathbf{P}_{#1}{\mathbb{C}}}
\newcommand{\interior}[1]{\mathring{#1}}

\newcommand{\CC}{\mathbb{C}}
\newcommand{\Chat}[1]{\widehat{\CC^{#1}}}
\newcommand{\cF}{\mathcal{F}}

\newcommand{\e}{\varepsilon}

\newcommand{\F}{\mathcal{F}}

\newcommand{\Hol}{\operatorname{Hol}}

\newcommand{\J}{\mathcal{J}}

\newcommand{\NN}{\mathbb{N}}
\newcommand{\N}{\mathbb{N}}

\newcommand{\RR}{\mathbb{R}}

\newcommand{\ZZ}{\mathbb{Z}}

\newtheorem{theorem}{Theorem}[section]
\newtheorem{corollary}[theorem]{Corollary}
\newtheorem{lemma}[theorem]{Lemma}

\theoremstyle{definition}
\newtheorem{definition}[theorem]{Definition}
\newtheorem{remark}[theorem]{Remark}
\newtheorem{example}[theorem]{Example}

\title[Non-bulging Fatou components]{Non-bulging Fatou components\\ for transcendental skew-products}
\author{Tom Potthink}
\address{Tom Potthink, Univ. Bordeaux, CNRS, Bordeaux INP, IMB, UMR 5251, F-33400 Talence, France}
\email{tom.potthink@math.u-bordeaux.fr}

\author{Jasmin Raissy}

\address{Jasmin Raissy, Univ. Bordeaux, CNRS, Bordeaux INP, IMB, UMR 5251, F-33400 Talence, France \& Institut Universitaire de France (IUF)}
\email{jasmin.raissy@math.u-bordeaux.fr}


\begin{document}

	\begin{abstract}
		In this paper, we investigate the bulging of escaping or oscillating Fatou components on invariant fibers for general skew-products, with a focus on the dependence on the perturbation. We show that any orbitally unbounded component is non-bulging for an appropriate choice of perturbation, whereas sufficiently well-behaved perturbations can render it bulging when the fiber is attracting. Our results highlight that bulging is influenced by more than just the dynamics on the fiber and in the one-dimensional coordinate, contrasting sharply with established results for non-escaping Fatou components.
	\end{abstract}
    \maketitle
	
	

\section{Introduction} \label{SECTION:1}

A central goal of complex dynamics is the study of the Fatou and Julia sets, which correspond, respectively, to the sets of stable and unstable behavior of the iterates of a holomorphic function. 

While the dynamics of entire functions in one complex variable is by now well understood, even simple extensions to higher dimensions, such as skew-products, exhibit fundamentally new phenomena that are not yet fully explored.

In dimension one, the types of components that constitute the Fatou set of an entire function are nearly completely classified. The basic classification was first established by Fatou and later further developed and completely proven by several authors (Julia, Leau, Siegel, Herman...). Furthermore, more recently, progress has been made to better understand wandering domains (cf. \cite{Beni2019} and \cite{ferreira2024classifyingmultiplyconnectedwandering}). In higher dimensions, however, no similarly complete understanding of Fatou components exists. A natural way to bridge one- and two-dimensional dynamics is to consider maps with a partially decoupled structure, where one coordinate evolves independently.

In this paper, we study two-dimensional Fatou components of \emph{skew-products}, i.e., maps $F\colon\CC^2\to\CC^2$ of the form
\begin{equation*}
    F(z,w) = (f(z,w), g(w)).
\end{equation*}
Skew-products serve as a natural stepping stone from one-dimensional to two-dimensional dynamics, since they often allow the reduction of two-dimensional behavior to one-dimensional analysis. Of particular interest is the behavior near an invariant fiber $\{w=c\}$, as on the fiber itself, the skew-product reduces to a one-dimensional function. In fact, the interplay between the dynamics on invariant fibers and the full two-dimensional system raises subtle questions about stability and the structure of Fatou components. This setting has been explored, for example, in \cite{Lilov_2004} for polynomial skew-products and \emph{attracting} fibers, that is, fibers $\{w=c\}$ where $c$ is an \emph{attracting} fixed point  of $g$ (i.e., $g(c)=c$ and $|g'(c)|<1$). 

Without loss of generality, we assume throughout that $c = 0$ and we write the skew-product $F$ as
\begin{equation*}
    F(z,w) = (f(z) + r(z,w), g(w))
\end{equation*}
where the holomorphic function $f$ gives the one-dimensional action of $F$ on the invariant fiber and $r\colon\CC^2\to\CC$ is a holomorphic map such that $r(z,0)=0$. We refer to $r$ as the \emph{perturbation} or \emph{error term}, since it measures the deviation from the one-dimensional dynamics on the fiber.  

Since our work concerns the Fatou set in higher dimensions, it is necessary to discuss which definition of normality is used. For polynomial skew-products, the choice of normality is clear; however, this is not the case in the transcendental setting. Building on the discussion in \cite{Arosio2019}, several viable extensions of the one-dimensional notion exist, depending on which compactification for $\CC^2$ is used. In this paper, we consider three options: the one-point compactification $\Chat{2}$, the projective space $\PC{2}$, and the product space $\Chat{} \times \Chat{}$. We show in \Autoref{SECTION:2} that for our results the choice of definition is immaterial.

Geometrically, when investigating the dynamics near an invariant fiber it is natural to study whether a stable one-dimensional region persists when embedded in the full two-dimensional system. More precisely, one may ask: when does a one-dimensional Fatou component extend to a two-dimensional one? In particular, we seek to understand whether all one-dimensional Fatou components of $f$ on the invariant fiber \emph{bulge}, that is, extend to two-dimensional Fatou components of $F$.  Bulging creates a direct link between the dynamics on the fiber and the full skew-product near the fiber, potentially reducing the analysis to the well-understood one-dimensional case. This setting and related phenomena have been studied extensively; see, e.g., \cite{Lilov_2004}, \cite{Peters_Vivas_2016}, \cite{Peters_Smit_2017}, \cite{Ji_2019}, and \cite{Ji_Shen_2025} (see also \cite{Raissy_2016} for an overview). In particular, Lilov showed in \cite{Lilov_2004} that periodic one-dimensional Fatou components with bounded orbits on an \emph{attracting} invariant  fiber always bulge. 

In this paper, we focus on \emph{escaping} or \emph{oscillating} Fatou components on the fiber, that is, components for which at least a subsequence of iterates tends uniformly to infinity. This includes \emph{escaping} and \emph{oscillating wandering domains}, \emph{Baker domains}, and \emph{basins of attraction at infinity}, (these last basins can arise if the skew-product restricts to a polynomial on the fiber). In these cases, we show that the bulging of a component depends sensitively on the perturbation, reflecting the fact that the perturbation may evolve as the orbit escapes to infinity. 

\eject
Our main result shows the existence of perturbations preventing stable behavior around points with unbounded orbits. The main idea is to construct a perturbation that forces nearby orbits to exhibit incompatible behaviors, some escaping to infinity while others remain bounded, thereby destroying normality.


\begin{theorem}\label{THEOREM:NONBULGING}
    Let $f$ and $g$ be nonconstant entire functions with $g(0) = 0$. Furthermore, assume that there exists $z_0 \in \CC$ with $f^{n_k}(z_0) \to \infty$ as $k\to\infty$ for some subsequence of iterates. Then there exists $h$ entire such that the point $(z_0, 0)$ belongs to the Julia set of the skew-product $F\colon \CC^2 \to \CC^2$ defined as
    \begin{equation}\label{eq:THEOREMNONBULG:Skew}
         F(z,w) =(f(z) + wh(z), g(w)).
    \end{equation}
\end{theorem}
\smallskip

    Observe that, for normality with respect to the compactifications $\PC{2}$ or $\Chat{2}$, but not $(\Chat{})^2$, the fact that $0$ is in the Julia set of $g$ does not automatically imply that $(z_0,0)$ is in the Julia set of $F$. In particular, our result remains nontrivial even when $0$ is in the Julia set of $g$.

As an immediate consequence, we obtain the following result about the existence of non-bulging Fatou components.

\begin{corollary}\label{COROLLARY:NONBULGING}
    Let $f$ and $g$ be nonconstant entire functions with $0$ a fixed point of $g$ belonging to the Fatou set of $g$, and $U$ a Fatou component of $f$. Furthermore assume that there exists $z_0 \in U$ with $f^{n_k}(z_0) \to \infty$ as $k\to\infty$ for some subsequence of iterates.
    Then there exists $h$ entire such that the Fatou component $U$ does not bulge for the skew-product $F$ as defined in \eqref{eq:THEOREMNONBULG:Skew}.
\end{corollary}

This statement holds for escaping wandering domains, oscillating wandering domains, Baker domains, and basins of attraction at infinity on either an attracting fiber or an invariant fiber in a Siegel disk. 
Escaping and oscillating domains arise only in the transcendental setting, thus our focus here is mainly on transcendental skew-products. The only exception is when the function reduces to a polynomial on the fiber, which always has a basin at infinity. In this case, the construction still applies and the resulting skew-product is in general non-polynomial. However, the basin at infinity does not need to bulge for polynomial perturbations, as shown in \Autoref{EXAMPLE:basininfinity}. This provides the only case where a polynomial skew-product has a non-bulging Fatou component on an attracting fiber, showing that basins at infinity for polynomial skew-products can be non-bulging. 

Our constructions constitute the first known example of non-bulging wandering domains and the second example of non-bulging Fatou components on attracting fibers in general, we achieved a similar construction for Baker domains in \cite{Bakerpaper_2025}, in collaboration with Benini. 

On the other hand, we also show that for attracting invariant fibers, wandering domains can bulge in non-trivial settings and we give a sufficient condition for bulging for a particular class of functions with wandering domains. These results are described in \Autoref{THEOREM:BULGINGORDER} and \Autoref{COROLLARY:GROWTH}. 

These results highlight a key difference with periodic Fatou components with bounded orbits: for orbitally unbounded components, the behavior depends sensitively on the interaction between the base dynamics and the perturbation.

\subsection*{Outline of the paper}
In \Autoref{SECTION:2} we review notions of normality in higher dimensions and present other preliminaries. In \Autoref{SECTION:4} we give the general construction of perturbations that prevent bulging proving \Autoref{THEOREM:NONBULGING}, while in \Autoref{SECTION:3} we study the bulging of wandering domains near an attracting invariant fiber, proving \Autoref{THEOREM:BULGINGORDER} and \Autoref{COROLLARY:GROWTH}.

\subsection*{Notation} 
We denote by $\pi_z\colon \CC^2 \to \CC, \pi_z(z,w) = z$ and $\pi_w\colon \CC^2 \to \CC, \pi_w(z,w) = w$ the projections onto the $z$ and $w$ coordinates respectively. 
Additionally, where there is no risk of confusion, we write $(z_k, w_k) := F^{k}(z,w)$ for $k \in \NN$, $(z,w) \in \CC^2$, and $F\colon \CC^2 \to \CC^2$.
Lastly, we use $\NN$ to denote the natural numbers including $0$.

\subsection*{Acknowledgments} We wish to thank Anna Miriam Benini for useful discussions. This project was supported in part with funding from the ANR PADAWAN /ANR-21-CE40-0012-01, ANR TIGerS/ANR-24-CE40-3604, the Institut Universitaire de France (IUF) and the French Italian University and Campus France through the PHC Galileo program, under the project ``From rational to transcendental: complex dynamics and parameter spaces''.


\section{Preliminaries}\label{SECTION:2}

First, we recall the following definitions and statements from one-dimensional dynamics.

    Let $f\colon \CC \to \CC$ be an entire function. 
    The \emph{Fatou set} of $f$ is defined as the set of normality of the iterates $\lbrace f^k \mid k \in \NN \rbrace$.
    The components of the Fatou set are called \emph{Fatou components} of $f$. Given $U \subset \CC$ a Fatou component of the entire function $f$, $f^n(U)$ is contained in some Fatou component $U_n$ of $f$ for every $n \in \NN$, and we call $U$
    \begin{enumerate}
        \item \emph{periodic} if there exists $n \in \NN$ such that $U = U_n$, and \emph{invariant} if this holds for $n = 1$,
        \item \emph{preperiodic} if $U_n = U_m$ for some $n \neq m$, and
        \item \emph{wandering} if $U_n \neq U_m$ for all $n \neq m$.
    \end{enumerate}

Fatou Classification Theorem gives a complete classification for (pre-)periodic Fatou components. In particular, if $U$ is an invariant Fatou component of $f$, then either:
        \begin{enumerate}
            \item[(i)] $U$ is the \emph{immediate attractive basin} of some $z_0 \in U$. In particular, it holds that $\lvert f'(z_0) \rvert < 1$, or
            \item[(ii)] $U$ is the \emph{parabolic domain} of some $z_0 \in \partial U$, or
            \item[(iii)] $U$ is a \emph{rotation domain}, that is, either a \emph{Siegel disk} around an elliptic fixed point $z_0 \in U$ or an \emph{Herman ring}, or
            \item[(iv)] $U$ is a \emph{Baker domain}, that is, $f^{n}(z) \to \infty$ for all $z \in U$ as $n \to \infty$, but $f(\infty)$ is not well-defined.         \end{enumerate}

Recall that only transcendental entire functions can have Baker domains.

The Non Wandering Domain Theorem due to  Sullivan \cite{sul}
states that for rational functions of degree $d\ge 2$ every Fatou component  is  eventually periodic. Therefore, for rational functions of degree $d\ge 2$, we have a complete description  of the  dynamics in the Fatou set.  On the other hand, prior to Sullivan's result, the first examples of an entire function with a wandering domains were found by Herman and published in~\cite{Baker1984}.
The well-understood one-dimensional dynamics of such functions, similar to $f(z) = z + 2 \pi i + e^z$, involve a wandering domain $U_0$ around an attracting fixed point $z_0$ of an associated function as described in the following lemma, that we shall use in \Autoref{SECTION:3} to find a class of skew-products with bulging wandering domains.

\begin{lemma}[{\cite[Section 5]{Baker1984}}]\label{LEMMA:WDExamplesin1D}
    Let $T \in \CC^*$, $p$ a $ T $-periodic entire function, and $f(z) = z + p(z) + T$.
    If the entire function $z + p(z)$ has an attracting fixed point $z_0$ with corresponding basin of attraction $U_0$, then the sets $U_n := U_0 + nT$, $n\in\ZZ$, are pairwise disjoint Fatou components of $f$ satisfying $f(U_n) \subset U_{n+1}$. In particular, they form a sequence of wandering domains for $f$.
\end{lemma}

We now discuss normality in the higher-dimensional setting. While in one dimension, it is natural to consider the compactification of $\CC$ given by the Riemann sphere $\widehat{\CC}$, the one point compactification of $\CC^2$ is not a complex manifold, and there are several natural compactifications of $\CC^2$ that are complex manifolds (see \cite{Morrow} and \cite{Brenton} for example). 
Several definitions of normality have been used in higher-dimensional dynamics.
Here, we consider definitions based on either the one-point compactification $\Chat{n} := \CC^n \cup \lbrace \infty \rbrace$, as first presented in \cite{wu1967normal}, the projective space $\PC{n}$, as in \cite{Arosio2019}, or $(\Chat{})^n = \Chat{} \times \dots \times \Chat{} $.
The last is of interest since it extends the dynamics of direct products in an intuitive manner.

\begin{definition}[{\cite[Definition 1.1]{wu1967normal}}]
    Let $X$ be a complex manifold and let $n\in\N$. 
    A family $ \cF \subset \Hol(X, \CC^n) $ of holomorphic maps is called \emph{$\Chat{n}$-normal} if every sequence in $\cF$ is either uniformly divergent to infinity on compact subsets or has a subsequence converging uniformly on compact subsets to a limit in $\Hol(X, \CC^n)$.
\end{definition}

\begin{definition}[{\cite[Definition 2.1]{Arosio2019}}]
    Let $X$ be a complex manifold and let $n\in\N$. 
    A family $ \cF \subset \Hol(X, \CC^n) $ of holomorphic maps is called \emph{$\PC{n}$-normal} if every sequence in $\cF$ has a subsequence converging uniformly on compact subsets to a limit in $\Hol(X, \PC{n})$.
\end{definition}


\begin{definition}
	Let $X$ be a complex manifold and let $n\in\N$. 
    A family $ \cF \subset \Hol(X, \CC^n) $ of holomorphic maps is called \emph{$(\Chat{})^n$-normal} if every sequence in $\cF$ has a subsequence converging uniformly on compact subsets to a limit in $\Hol(X, (\Chat{})^n)$.
\end{definition}


Using the definitions above, the Fatou set of a holomorphic function in higher dimensions is defined as the set where the iterates form a normal family.

\begin{definition}
    Let $F\colon \CC^n \to \CC^n $ be a holomorphic function for some $n \in \N$. We define the $\Chat{n}$-, $\PC{n}$-, and $(\Chat{})^n$-\emph{Fatou set} of $F$ as the respective sets of normality of the iterates $(F^n)_{n\in\NN}$. 
    The corresponding Julia sets are defined as the complements.
\end{definition}

The Fatou sets in higher dimensions share some properties with their one-dimensional counterpart, though not all. 
While the Fatou set remains backward invariant, a property which we will use later, it is no longer forward invariant. 
For $\Chat{n}$, the lack of forward invariance was discussed in \cite[Proposition 2.4]{Fornss1998FatouAJ}. 
The proof of backward invariance follows the same argument as in the one-dimensional case.

\begin{lemma}\label{LEMMA:BackwardInvariance}
    For any holomorphic map $F\colon \CC^n \to \CC^n$,  the Fatou set is backward invariant:
    \begin{equation*}
        F^{-1}(\F(F)) \subset \F(F).
    \end{equation*} 
\end{lemma}

%

While the three notions of normality are not equivalent, they satisfy an inclusion relation: normality with respect to $\PC{n}$ or $(\Chat{})^n$ implies normality with respect to $\Chat{n}$. 
This is described in the following lemma; the $\PC{n}$ case was proven in \cite[Lemma 2.4]{Arosio2019} and the $(\Chat{})^n$ case follows analogously.

\begin{lemma}[{\cite[Lemma 2.4]{Arosio2019}}]\label{Lemma:NormalityRelationsBakerVariant}
	Let $X$ be a connected complex manifold and let $n\in\N$. If a family $\mathcal{F} \subset \Hol(X, \CC^n)$ is $\PC{n}$-normal or $(\Chat{})^n$-normal, then it is already $\widehat{\CC^n}$-normal. 
    As a consequence, for each of these three compactifications, any limit function either lies in $\Hol(X, \CC^n)$ or the corresponding sequence diverges locally uniformly to infinity.
\end{lemma}  

Our main result \Autoref{THEOREM:NONBULGING} establishes nonbulging for $\Chat{n}$-normality, which by the above lemma suffices for all three notions of normality.
Hence, the choice of normality is immaterial for the rest of this paper.
When the type of Fatou set is not explicitly specified in a statement, it applies to all three definitions. 
In particular, we use $\F(F)$ and $\J(F)$ to denote the Fatou and Julia sets of $F$ regardless of which notion of normality is used. 

Lastly, the following definition provides the precise meaning of bulging. 

\begin{definition}\label{DEF:Bulging}
    Given $f, g: \CC \to \CC$ and $h\colon \CC^2 \to \CC$ holomorphic functions with $0$ a fixed point in the Fatou set of $g$, consider the skew-product $F\colon\CC^2\to\CC^2$ of the form $F(z,w) = (f(z) + wh(z,w), g(w))$. 
    We call a Fatou component $U \subset \CC$ of $f$ \emph{bulging} with respect to a given compactification (i.e., a notion of normality) if $U \times \lbrace 0 \rbrace$ lies in the corresponding Fatou set of $F$.
\end{definition}

%
%


\section{Non-bulging Fatou Components}\label{SECTION:4}

In this section, we construct perturbations that prevent bulging. The strategy is to carefully design an entire function using approximation techniques so that nearby orbits behave in fundamentally different ways. The main tool is an approximation argument based on the following two well-known results, namely a variant of Runge's Theorem and Weierstraß's Theorem for uniformly convergent sequences. 

\begin{lemma}[Runge, {\cite[p. 460]{Eremenko_Ljubich_1987}; see also \cite[Lemma 5.1]{Beni2019}}]\label{Lemma:Runge}
	Let $(E_n)_{n\in\NN}$ be a sequence of compact subsets of $\CC$ such that
	\begin{enumerate}
		\item the sets $\CC \setminus E_n$ are connected for all $n\in\NN$,
		\item the sets are pairwise disjoint, i.e., $E_n \cap E_m = \emptyset$ for all $n\neq m$, and
		\item $\min\lvert E_n \rvert := \min\set{\lvert z \rvert}{z \in E_n} \to \infty$ as $n \to \infty$.
	\end{enumerate}
	Furthermore, suppose that $\psi$ is holomorphic in a neighborhood of $E = \bigcup_{n=0}^{\infty} E_n$, and let $(\e_n)_{n\in\NN}$ be a sequence of positive real numbers. Then there exists an entire function $h$ such that
	\begin{equation*}
		\lvert h(z) - \psi(z) \rvert < \e_n
	\end{equation*}
	for all $z \in E_n$ and $n\in\NN$.
\end{lemma}

\begin{theorem}[Weierstraß, {\cite[p. 176]{Ahlfors_1979}}]\label{theorem:weier}
    Let $G_1 \subset G_2 \subset \dots$ be an increasing sequence of domains in $\CC$ and $(f_n)_{n\in\NN}$ a sequence of holomorphic functions $f_n\colon G_n \to \CC$. Let $G := \bigcup_{n\in\NN} G_n$ and assume that $(f_n)_{n\in\NN}$ converges uniformly on any compact subset of $G$ to a limit function $f$. Then $f$ is holomorphic on $G$.
\end{theorem}

We now prove \Autoref{THEOREM:NONBULGING}. The key step in the proof is to  construct a perturbation $h$ that forces orbits near to the wandering domain in the invariant fiber to behave in incompatible ways, as one orbit grows large while another remains bounded, violating normality.

\begin{proof}[Proof of Theorem~{\ref{THEOREM:NONBULGING}}]
    By assumption, there exists a sequence $(z_n)_{n\in\NN}$ in $\CC$ satisfying $f(z_n) = z_{n+1}$ for all $n\in\NN$ with a subsequence $(z_{n_k})_{k\in\NN}$ diverging to infinity. By choosing a different subsequence, we may assume that $\lvert z_j \rvert < \lvert z_{n_k} \rvert$ for all $j < n_k$ and that $n_0 = 0$.


    Given any entire function $h$, define the skew-product
    \begin{equation*}
        F_h (z,w) = (f(z) + w h(z), g(w)).
    \end{equation*}
    Our goal is to show that, for an appropriately chosen $h$, the point $(z_0, 0)$ belongs to the Julia set of $F_h$. The construction of $h$ is provided by the following lemma. 

    \begin{lemma}\label{PROOF:OSC:LEMMA1}
        Let $(z_{n_k})_{k\in\NN}$ be the subsequence from above. There exist an entire function $h\colon \CC \to \CC$ and two sequences $(w_k)_{k \in \NN}$ and $(\widetilde{w}_k)_{k \in \NN}$ of nonzero complex numbers converging to $0$ with $\lvert g^{n_k+1}(w_k) \rvert < 1$ and $\lvert g^{n_k+1}(\widetilde{w}_k) \rvert < 1$ such that the skew-product $F_h$ satisfies
        \begin{equation}\label{PROOF:OSC:LEMMA1:EQ1}
            \lvert \pi_z F_h^{n_k+1}(z_0, w_k) \rvert \geq k
        \end{equation}
        and
        \begin{equation}\label{PROOF:OSC:LEMMA1:EQ2}
            \lvert \pi_z F_h^{n_k+1}(z_0, \widetilde{w}_k) \rvert \leq 1
        \end{equation}
        for all $k\in\NN$.
    \end{lemma}

    We now prove that $(z_0, 0)$ belongs to the Julia set $\J(F)$ of $F = F_h$, where $h$ is as provided by \Autoref{PROOF:OSC:LEMMA1}.
    Assume by contradiction that $(z_0, 0)$ belongs to the Fatou set $ \F(F)$ of $F$. 
    Then, the iterates of $F$ form a normal family at $(z_0, 0)$. 
    In particular, one can find a compact neighborhood $K$ of $(z_0,0)$ and a subsequence $(F^{n_{k(j)} + 1})_{j\in\NN}$ of iterates such that this subsequence either diverges uniformly to $\infty$ on $K$ or converges uniformly to a finite limit function on $K$.

    Assume that the subsequence diverges uniformly to infinity. In particular, there exists $J \in\NN$ such that for all $j \geq J$ and $(z,w) \in K$,
    \begin{equation*}
        \lVert F^{n_{k(j)} + 1}(z,w) \rVert \geq 2,
    \end{equation*}
    where $\lVert \cdot \rVert$ denotes the Euclidean norm on $\CC^2$.
    Since $\widetilde{w}_k \to 0$ as $k \to \infty$ and $K$ is a neighborhood of $(z_0,0)$, there exists $J_K' \in \NN$ such that $(z_0, \widetilde{w}_{k(j)}) \in K$ for all $j \geq J_K'$. 
    Using $\lvert g^{n_{k(j)} + 1}(\widetilde{w}_{k(j)}) \rvert < 1$ and \eqref{PROOF:OSC:LEMMA1:EQ2}, it follows that
    \begin{align*}
        \lVert F^{n_{k(j)} + 1}(z_0,\widetilde{w}_{k(j)}) \rVert &\leq \lvert \pi_z F^{n_{k(j)} + 1}(z_0,\widetilde{w}_{k(j)}) \rvert + \lvert g^{n_{k(j)} + 1}(\widetilde{w}_{k(j)}) \rvert \\
        &< 1 + 1 = 2
    \end{align*}
    for all $j \geq \max\lbrace J,J_K' \rbrace$, giving a contradiction.

    Now, assume that the subsequence of iterates has a finite limit function $\Phi\colon K \to \CC^2$, i.e., $F^{n_{k(j)} + 1}$ converges to $\Phi$ uniformly on $K$ as $j \to \infty$. Then there exists $J \in \NN$ such that
    \begin{equation*}
        \lVert F^{n_{k(j)} + 1}(z,w) - \Phi(z,w) \rVert \leq 1
    \end{equation*}
    for all $j \geq J$ and $(z,w) \in K$. As before, we find $J_K'\in\NN$ such that $(z_0, w_{k(j)}) \in K$ for all $j \geq J_K'$. 
    Let $A := \max_{(z,w) \in K} \lVert \Phi(z,w) \rVert$ and choose $J''\in\NN$ such that $k(j) > A + 1$ for all $j \geq J''$.

    Using \eqref{PROOF:OSC:LEMMA1:EQ1} and the reverse triangle inequality, it follows that
    \begin{align*}
        \lVert F^{n_{k(j)} + 1}(z_0,w_{k(j)}) - \Phi(z_0,w_{k(j)}) \rVert &\geq \left\lvert \lVert F^{n_{k(j)} + 1}(z_0,w_{k(j)}) \rVert - \lVert \Phi(z_0,w_{k(j)}) \rVert \right\rvert \\
        %
        %
        &\geq k(j) - A \\
        &> A + 1 - A = 1
    \end{align*}
    for all $j \geq \max\lbrace J,J_K',J'' \rbrace$, giving a contradiction. 
    Therefore $(z_0,0)$ does not belong to the Fatou set of $F$.
\end{proof}

    \begin{proof}[Proof of Lemma~{\ref{PROOF:OSC:LEMMA1}}]
        The function $h$ will be defined as the limit of an inductively constructed sequence of entire functions.  

        Fix
        \begin{equation*}
            0 < \widetilde{\delta}_0 < \min\left\lbrace \frac{\lvert z_{n_{1}} \rvert - \lvert z_{n_0} \rvert}{4}, 1 \right\rbrace
        \end{equation*}
        and, for each $k \geq 1$, fix
        \begin{equation*}
            0 < \widetilde{\delta}_k < \min\left\lbrace \frac{\inf_{j < n_k}( \lvert z_{n_k} \rvert - \lvert z_{j} \rvert )}{4}, \frac{\lvert z_{n_{k+1}} \rvert - \lvert z_{n_k} \rvert}{4}, \frac{\widetilde{\delta}_{k-1}}{2} \right\rbrace.
        \end{equation*}
        This yields a strictly decreasing sequence $(\widetilde{\delta}_k)_{k\in\NN}$ of positive real numbers such that for every sequence $(\delta_k)_{k\in \NN}$ of positive real numbers that is bounded from above by $(\widetilde{\delta}_k)_{k\in \NN}$ (i.e., such that $0 < \delta_k < \widetilde{\delta}_k$), the following hold:
        \begin{enumerate}[label=(\alph*)]
            \item $\delta_k \to 0$ as $k\to \infty$,
            \item the sets $D_k := \overline{D(z_{n_k}, \delta_k)}$ are pairwise disjoint, compact, and simply connected,
            \item setting $H_0 := \emptyset$ and $H_k := \overline{D(0, \lvert z_{n_k} \rvert - 2\delta_k)}$ for $k \geq 1$, the sets $H_k$ are compact, simply connected, form an exhaustion of $\CC$, and satisfy $H_k \subset \interior{H}_{k+1}$, and
            \item for every $k \in \NN$, the set $H_k$ contains $D_j$ for all $j < k$, is disjoint from $D_j$ for all $j \geq k$, and satisfies $z_j \in \interior{H_k}$ for all $j < n_k$.
        \end{enumerate}
        \begin{figure}[htbp]
            \centering
            \begin{tikzpicture}
                \filldraw[fill=gray!20, draw=black]
                    (0,0) circle (3);
                \node[below] at (0,3) {$H_1$};

                \filldraw[fill=gray!30, draw=black]
                    (5,0) circle (1);
                \node[below] at (5,1) {$D_1$};

                \draw[-]
                    (5, 0) -- (4,0);
                \node[below] at (4.5,0) {$\delta_1$};
                \draw[|-|]
                    (3, 0) -- (4,0);
                \node[below] at (3.5,0) {$\delta_1$};

                \filldraw[fill=gray!30, draw=black]
                    (0,0) circle (1);
                \draw[-]
                    (-1, 0) -- (0,0);
                \node[below] at (-0.5,0) {$\delta_0$};
                \node[below] at (0,1) {$D_0$};
                \node at (0,0) [circle, fill=black, inner sep=1pt,  label=right:$z_{0}$] {};
                \node at (-2,1) [circle, fill=black, inner sep=1pt,  label=right:$z_{1}$] {};
                \node at (1.5,-1.5) [circle, fill=black, inner sep=1pt,  label=right:$z_{2}$] {};
                \node at (-1.5,-1.25) [circle, fill=black, inner sep=1pt,  label=right:$z_{3}$] {};
                \node at (5,0) [circle, fill=black, inner sep=1pt,  label=right:$z_{4}$] {};
            \end{tikzpicture}
            \caption{The first sets of the construction if $n_0 = 0$ and $n_1 = 4$.}
            \label{fig:ConstructionNonBulg}
        \end{figure}
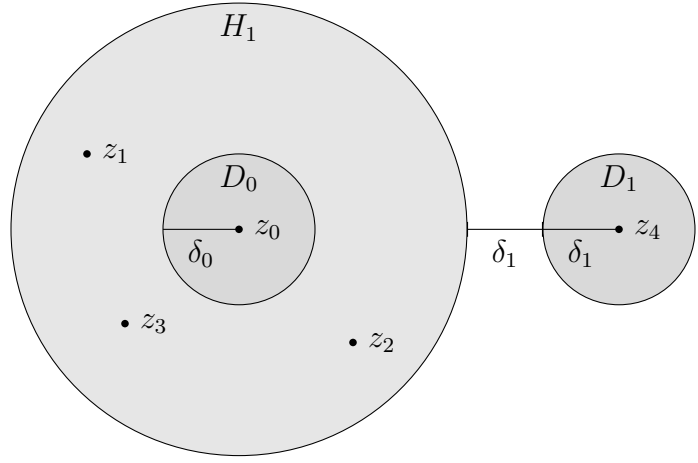

        The following lemma provides the core inductive construction. Each of the conditions is designed to control the behavior of selected orbits and ensure stability under small perturbations.
        
                \begin{lemma}\label{PROOF:OSC:LEMMA2}
            There exist a sequence of nonconstant entire functions $(h_k)_{k\in\NN}$, two sequences $(w_k)_{k \in \NN}$ and $(\widetilde{w}_k)_{k \in \NN}$ of nonzero complex numbers, and two sequences of positive real numbers $(R_k)_{k\in\NN}$ and $(\delta_k)_{k\in\NN}$ in the interval $\left( 0,\frac{1}{3} \right)$ such that $(\delta_k)_{k\in \NN}$ is bounded from above by $(\widetilde{\delta}_k)_{k\in\NN}$, and for all $k \geq 1$, the following hold:
            \begin{enumerate}[label=(\roman*)]
                \item $ \lvert w_k \rvert < \frac{1}{k} $, $\lvert \widetilde{w}_k \rvert < \frac{1}{k}$, $ \lvert g^j(w_k) \rvert < 1 $, and $\lvert g^j(\widetilde{w}_k) \rvert < 1$ for all $j \in \lbrace 0,\dots,n_k + 1 \rbrace$, \label{PROOF:OSC:LEMMA2:PROP1}
                \item $\Delta_k := \overline{D(\pi_z F_{h_{k-1}}^{n_k}(z_0, w_k), R_k)} \subset D_k $ and $\widetilde{\Delta}_k := \overline{D(\pi_z F_{h_{k-1}}^{n_k}(z_0, \widetilde{w}_k), R_k)} \subset D_k$ are disjoint, \label{PROOF:OSC:LEMMA2:PROP2}
                \item $\displaystyle\max_{z \in H_k}\lvert h_{k}(z) - h_{k-1}(z) \rvert < \frac{1}{2^{k+1}} \min\lbrace \delta_0, \dots, \delta_{k} \rbrace$, \label{PROOF:OSC:LEMMA2:PROP3}
                \item $\displaystyle\max_{z \in \Delta_k}\left\lvert h_{k}(z) + \frac{1}{g^{n_k}(w_k)} \left( f(\pi_z F_{h_{k-1}}^{n_k}(z_0, w_k)) - (k+1) \right) \right\rvert < \frac{1}{3}$, \label{PROOF:OSC:LEMMA2:PROP4}
                \item $\displaystyle\max_{z \in \widetilde{\Delta}_k}\left\lvert h_{k}(z) + \frac{1}{g^{n_k}(\widetilde{w}_k)} f(\pi_z F_{h_{k-1}}^{n_k}(z_0, \widetilde{w}_k))  \right\rvert < \frac{1}{3}$,\label{PROOF:OSC:LEMMA2:PROP5}
                \item for any entire function $\widetilde{h}\colon \CC \to \CC$ with $\displaystyle\max_{z \in H_k}\lvert h_{k-1}(z) - \widetilde{h}(z) \rvert < \delta_k$, we have
                \begin{equation*}
                    \lvert \pi_z F_{h_{k-1}}^{n_k}(z_0, w_k) - \pi_z F_{\widetilde{h}}^{n_k}(z_0, w_k) \rvert < R_k
                \end{equation*}
                and
                \begin{equation*}
                    \lvert \pi_z F_{h_{k-1}}^{n_k}(z_0, \widetilde{w}_k) - \pi_z F_{\widetilde{h}}^{n_k}(z_0, \widetilde{w}_k) \rvert < R_k,
                \end{equation*} \label{PROOF:OSC:LEMMA2:PROP6}
                \item for all $z$ and $z' $ belonging either both to $ \Delta_k$ or both to $\widetilde{\Delta}_k$, we have
                \begin{equation*}
                    \lvert f(z) - f(z') \rvert < \frac{1}{3}.
                \end{equation*} \label{PROOF:OSC:LEMMA2:PROP7}
            \end{enumerate}
        \end{lemma}
        
Roughly speaking, the construction ensures that two nearby initial conditions lead to drastically different growth behaviors under iteration. In fact, \ref{PROOF:OSC:LEMMA2:PROP2} ensures separation of orbits, \ref{PROOF:OSC:LEMMA2:PROP4} and \ref{PROOF:OSC:LEMMA2:PROP5} enforce different growth behaviors, while \ref{PROOF:OSC:LEMMA2:PROP6} provides stability under perturbations.
        
    \begin{proof}[Proof of Lemma~{\ref{PROOF:OSC:LEMMA2}}]
        First, note that for all $k\in\NN$, for all $j < n_k$, and for all $\tilde{h}\colon \CC \to \CC$, entire we have
        \begin{equation}\label{PROOF:OSC:EQ:ZeroSpecialCase}
            \pi_z F^{j}_{\widetilde{h}}(z_0, 0) = \pi_z (f^{j}(z_0),0) = f^{j}(z_0) = z_{j} \in \interior{H_k}.
        \end{equation}
        
        To start the construction, let $w_0$ and $ \widetilde{w}_0$ be any nonzero complex numbers and set $h_0 := \operatorname{id}_{\CC}$.
        Moreover, choose $ 0 < R_0 = \delta_0 < \min\lbrace \widetilde{\delta}_0, \frac{1}{4} \rbrace $. 

        Now, for $k \geq 1$, given any entire nonconstant functions $h_0, \dots, h_{k-1}$, any complex numbers $w_0, \dots, w_{k-1}$, $\widetilde{w}_0, \dots, \widetilde{w}_{k-1}$, any positive real numbers $\delta_0, \dots, \delta_{k-1}$ bounded from above by $\widetilde{\delta}_0, \dots, \widetilde{\delta}_{k-1}$ and any positive real numbers $R_0, \dots, R_{k-1}$ satisfying pro\-per\-ties~\ref{PROOF:OSC:LEMMA2:PROP1}--\ref{PROOF:OSC:LEMMA2:PROP7}, we construct $h_k$, $w_k$, $\widetilde{w}_k$, $R_k$, and $\delta_k$ satisfying the same properties as follows.

        Thanks to \eqref{PROOF:OSC:EQ:ZeroSpecialCase}, the continuity of $w \mapsto \pi_z F^j_{h_{k-1}}(z_0,w) $ at $0$ for $j = 0, \dots, n_k$ allows us to find $w_{k}$ and $ \widetilde{w}_{k} $ nonzero complex numbers small enough such that
        \begin{itemize}
            \item $\lvert w_{k} \rvert < \frac{1}{k}$, $\lvert \widetilde{w}_{k} \rvert < \frac{1}{k}$, $ \lvert g^j(w_k) \rvert < 1 $, and $\lvert g^j(\widetilde{w}_k) \rvert < 1$ for all $j \in \lbrace 0,\dots,n_k + 1 \rbrace$, which is required for property~\ref{PROOF:OSC:LEMMA2:PROP1},
            \item $\pi_z F^{j}_{h_{k-1}}(z_0, w_{k})  \in \interior{H_k}$ for all $j < n_k$,
            \item $\pi_z F^{j}_{h_{k-1}}(z_0, \widetilde{w}_{k}) \in \interior{H_k}$ for all $j < n_k$,
            \item $\pi_z F^{n_k}_{h_{k-1}}(z_0, w_{k}) \in \interior{D_k}$, 
            \item $\pi_z F^{n_k}_{h_{k-1}}(z_0, \widetilde{w}_{k}) \in \interior{D_k}$, and
            \item $ \pi_z F^{n_k}_{h_{k-1}}(z_0, w_{k}) \neq \pi_z F^{n_k}_{h_{k-1}}(z_0, \widetilde{w}_{k})$.
        \end{itemize}
        It is possible to obtain the last property since both $h_{k-1}$ and $g$ are nonconstant functions.

        Choose $0<R_k <\frac{1}{3}$ small enough such that
        \begin{itemize}
            \item $\Delta_k := \overline{D(\pi_z F_{h_{k-1}}^{n_k}(z_0, w_k), R_k)} \subset D_k $ and $\widetilde{\Delta}_k := \overline{D(\pi_z F_{h_{k-1}}^{n_k}(z_0, \widetilde{w}_k), R_k)} \subset D_k$ are disjoint, and
            \item for all $z$ and $z' $ belonging either both to $ \Delta_k$ or both to $\widetilde{\Delta}_k$, we have
                \begin{equation*}
                    \lvert f(z) - f(z') \rvert < \frac{1}{3}.
                \end{equation*}
        \end{itemize}
        This yields properties~\ref{PROOF:OSC:LEMMA2:PROP2} and \ref{PROOF:OSC:LEMMA2:PROP7}.

        Using the continuity of the map $h \mapsto F_h$ and the fact that $\pi_z F^{j}_{h_{k-1}}(z_0, w_{k})  \in \interior{H_k}$ for all $j < n_k$, we can find a sufficiently small $\rho > 0$ such that the inequality $\max_{z\in H_{k}}\lvert h_{k-1}(z) - \widetilde{h}(z) \rvert < \rho$ implies
        \begin{equation*}
            \lvert \pi_z F_{h_{k-1}}^{n_k}(z_0, w_k) - \pi_z F_{\widetilde{h}}^{n_k}(z_0, w_k) \rvert < R_k.
        \end{equation*}
        Similarly, we find $\widetilde{\rho}>0$ such that the inequality $\max_{z\in H_{k}}\lvert h_{k-1}(z) - \widetilde{h}(z) \rvert < \widetilde{\rho}$ implies
        \begin{equation*}
            \lvert \pi_z F_{h_{k-1}}^{n_k}(z_0, \widetilde{w}_k) - \pi_z F_{\widetilde{h}}^{n_k}(z_0, \widetilde{w}_k) \rvert < R_k.
        \end{equation*}
        Setting $\delta_k := \min\lbrace \rho, \widetilde{\rho}, \widetilde{\delta}_k, \frac{1}{4} \rbrace$, we obtain property~\ref{PROOF:OSC:LEMMA2:PROP6}.

        Define the holomorphic function $\Phi\colon H_{k} \cup \Delta_{k} \cup \widetilde{\Delta}_{k} \to \CC$ as
        \begin{equation*}
            \Phi(z)=  \begin{cases}
                    h_{k-1}(z),&\;\text{ if}\; z\in H_{k},\\
                    - \frac{1}{g^{n_k}(w_k)} (f(\pi_z F_{h_{k-1}}^{n_k}(z_0, w_k)) - (k+1)),&\;\text{ if}\;z \in \Delta_{k}, \\
                    - \frac{1}{g^{n_k}(\widetilde{w}_k)} f(\pi_z F_{h_{k-1}}^{n_k}(z_0, \widetilde{w}_k)),&\;\text{ if}\;z \in \widetilde{\Delta}_{k}.
                \end{cases}
        \end{equation*}
        Since $H_k$, $\Delta_{k}$, and $\widetilde{\Delta}_{k}$ are compact, simply connected and disjoint, their union $H_k \cup \Delta_k \cup \widetilde{\Delta}_{k}$ is compact with connected complement. Therefore, we can apply Runge's Theorem \ref{Lemma:Runge}, which yields a nonconstant entire function $h_k\colon \CC \to \CC$ satisfying
        \begin{equation*}
            \max_{z \in H_k \cup \Delta_{k} \cup \widetilde{\Delta}_{k}}\left\lvert \Phi(z) - h_k(z) \right\rvert < \frac{1}{2^{k+1}} \min\lbrace \delta_0, \dots, \delta_{k} \rbrace < \frac{1}{3}.
        \end{equation*}
        Thus, we obtain properties~\ref{PROOF:OSC:LEMMA2:PROP3}--\ref{PROOF:OSC:LEMMA2:PROP5}. 
        This completes the construction step.
    \end{proof}
        \noindent\textit{Proof of Lemma~{\ref{PROOF:OSC:LEMMA1}} (continued).}
        We now use the sequences constructed in \Autoref{PROOF:OSC:LEMMA2}. First, we write $h_k$ as a telescopic sum. Fix any $m \in \NN$. Then, for all $k > m$, we have
        \begin{equation*}
            h_k = h_m +  \sum_{j = m}^{k-1} \left( h_{j + 1} - h_{j} \right).
        \end{equation*}
       Property~\ref{PROOF:OSC:LEMMA2:PROP3} yields
        \begin{equation*}
            \sum_{j = m}^{\infty} \max_{z\in H_{m}}\lvert h_{j + 1}(z) - h_{j}(z) \rvert \leq \sum_{j = m}^{\infty} \max_{z\in H_{j+1}}\lvert h_{j + 1}(z) - h_{j}(z) \rvert < \sum_{j = m}^{\infty} \frac{1}{2^{j+2}} < \infty.
        \end{equation*}
        Therefore, the telescopic sum forms a Cauchy sequence on $H_m$ and so it has a uniform limit function $\lim_{k \to \infty} h_k = h_m + \sum_{j = m}^{\infty} (h_{j + 1} - h_{j})$ on $H_m$. 

        Since the compact sets $H_k$ are an increasing sequence exhausting $\CC$, thanks to Weierstraß's Theorem~\ref{theorem:weier}, we obtain an entire limit function $h = \lim_{k\to \infty}h_k$. Moreover,   for any $k \in \NN$, the limit function satisfies
        \begin{equation}\label{PROOF:OSC:LEMMA1:EQ4}
            \begin{split}
                \max_{z \in H_{k+1}}\lvert h(z) - h_k(z) \rvert &\leq \ \sum_{j = k}^{\infty} \ \max_{z\in H_{j+1}}\lvert h_{j + 1}(z) - h_{j}(z) \rvert 
                \leq \sum_{j = k+1}^{\infty} \frac{1}{2^{j+1}} \min\lbrace \delta_0, \dots, \delta_j \rbrace \\
                &\leq \sum_{j = k+1}^{\infty} \frac{1}{2^{j+1}} \delta_{k+1} < \delta_{k+1}.
            \end{split}
        \end{equation}

       Our construction of $h$ ensures that the map $F_h(z,w) = (f(z) + wh(z), g(w))$, together with the sequences $(w_k)_{k\in\NN}$ and $(\widetilde{w}_k)_{k\in\NN}$, satisfies the assertions of \Autoref{PROOF:OSC:LEMMA1}.

        In fact, by \Autoref{PROOF:OSC:LEMMA2} property~\ref{PROOF:OSC:LEMMA2:PROP1}, both $(w_k)_{k\in\NN}$ and $(\widetilde{w}_k)_{k\in\NN}$ are sequences of nonzero complex with $\lim_{k\to\infty} w_k = \lim_{k\to\infty} \widetilde{w}_k = 0$. 

        Next, fix any $k\in\NN$. Thanks to the definitions of $\Delta_k$ and $\widetilde{\Delta}_k$, equation~\eqref{PROOF:OSC:LEMMA1:EQ4} for $k - 1$, and property~\ref{PROOF:OSC:LEMMA2:PROP6}, $h$ is sufficiently close to the recursively defined function $h_k$ to satisfy both $ \pi_z F_h^{n_k}(z_0, w_k) \in \Delta_k$ and $ \pi_z F_h^{n_k}(z_0, \widetilde{w}_k) \in \widetilde{\Delta}_k$. Thus, we can write
        \begin{equation*}
            \label{PROOF:OSC:LEMMA2:CONC_EQ1}
            \begin{split}
                |\pi_z F_h^{n_k+1}(z_0, w_k) &- (k+1)| \\
                &= \lvert f(\pi_z F_h^{n_k}(z_0, w_k)) - (k+1) + g^{n_k}(w_k)h(\pi_z F_h^{n_k}(z_0, w_k)) \rvert \\
                &\leq \left\lvert f(\pi_z F_h^{n_k}(z_0, w_k)) - f(\pi_z F_{h_{k-1}}^{n_k}(z_0, w_k)) \right\rvert \\
                & \qquad\qquad + \left\lvert g^{n_k}(w_k)h(\pi_z F_h^{n_k}(z_0, w_k)) + f(\pi_z F_{h_{k-1}}^{n_k}(z_0, w_k)) - (k+1) \right\rvert \\
                &\leq \frac{1}{3} + \left\lvert g^{n_k}(w_k) \right\rvert\left\lvert h(\pi_z F_h^{n_k}(z_0, w_k)) + \frac{f(\pi_z F_{h_{k-1}}^{n_k}(z_0, w_k)) - (k+1)}{\left\lvert g^{n_k}(w_k) \right\rvert} \right\rvert \\
                &\leq \frac{1}{3} + \max_{z \in \Delta_k}\left\lvert h(z) + \frac{f(\pi_z F_{h_{k-1}}^{n_k}(z_0, w_k)) - (k+1)}{g^{n_k}(w_k)}  \right\rvert \\
                &\leq \frac{1}{3} + \max_{z \in \Delta_k}\left\lvert h_{k}(z) + \frac{f(\pi_z F_{h_{k-1}}^{n_k}(z_0, w_k)) - (k+1)}{g^{n_k}(w_k)}  \right\rvert+ \max_{z \in \Delta_k}\left\lvert h(z) - h_k(z) \right\rvert \\
                &\leq \frac{1}{3} + \frac{1}{3} + \max_{z \in H_{k+1}}\left\lvert h(z) - h_k(z) \right\rvert \leq \frac{1}{3} + \frac{1}{3} + \delta_{k+1} < 1
            \end{split}
        \end{equation*}
        where we used $\pi_z F_h^{n_k}(z_0, w_k) \in \Delta_k$ and \eqref{PROOF:OSC:LEMMA1:EQ4} together with properties~\ref{PROOF:OSC:LEMMA2:PROP1}, \ref{PROOF:OSC:LEMMA2:PROP4}, and \ref{PROOF:OSC:LEMMA2:PROP7}. 
        We deduce that $\pi_z F_h^{n_k+1}(z_0, w_k) \in D(k + 1, 1)$ and in particular,
        \begin{equation*}
            \lvert \pi_z F_h^{n_k+1}(z_0, w_k) \rvert > k.
        \end{equation*}
        Arguing analogously for $|\pi_z F_h^{n_k+1}(z_0, \widetilde{w}_k)|$ and applying \ref{PROOF:OSC:LEMMA2:PROP5} instead of \ref{PROOF:OSC:LEMMA2:PROP4}, we obtain
        \begin{equation*}
            \begin{split}
                \lvert \pi_z F_h^{n_k+1}(z_0, \widetilde{w}_k) \rvert < 1,
            \end{split}
        \end{equation*}
        concluding the proof of \Autoref{PROOF:OSC:LEMMA1}.
    \end{proof}

\begin{remark}
    Note that a priori, the entire function $h$ may be transcendental, and it is reasonable to ask whether for polynomial skew-products, the basin of attraction at infinity always bulges. 
    However, the following example shows that the basin at infinity does not need to bulge for polynomial perturbations.
\end{remark}

\begin{example}\label{EXAMPLE:basininfinity}
    The basin at infinity of the polynomial skew-product $F\colon\CC^2\to\CC^2$
    \begin{equation*}
        F(z,w) = (z^2 - wz^3, \lambda w),
    \end{equation*}
    is non-bulging for any real number $ 0 < \lambda < 1 $.
    Indeed, let $x_0 > 0$ be such that $\lambda x_0 > 1$ and $\lambda^2 x_0 > 1$. 
    In particular, since $x_0 > 1$, the point lies in the basin at infinity of $z \mapsto z^2$. 
   Since $F$ preserves $\RR^2$, it suffices to show that there are points arbitrarily close to $(x_0,0)$ whose orbits are bounded and hence cannot diverge to infinity. Therefore, the basin at infinity cannot bulge.

    We first show that for every $\delta > 0$ there exists $0 < y_0 < \delta$ and $n_0 \in \NN$ such that $x_{n_0} > 0$ and $x_{n_0+1} \leq 0$, where $(x_{n_0},y_{n_0}) := F^{n_0}(x_0,y_0)$. 

    Let $\delta > 0$.
    By our choice of $x_0$, we have  $\lim_{n\to\infty} ( \lambda^2 x_0 )^{2^n} = \infty$ and so there exists $N \in \NN$ with
    \begin{equation}\label{eq:polynomial:1}
        \frac{1}{( \lambda^2 x_0 )^{2^{n-1}}} < 1 - \lambda
    \end{equation}
    for all $n \geq N$. Fix $n_0 \geq N$ large enough so that
    \begin{equation*}
        y_0 := \frac{1}{ (\lambda x_0)^{2^{n_0}} \lambda^{n_0-1}} < \delta.
    \end{equation*}
    Using \eqref{eq:polynomial:1} and the definition of $y_0$, we deduce
    \begin{equation*}
        x_0^{2^{n_0-1}}\lambda^{n_0-1} y_0 = \frac{1}{( \lambda^2 x_0 )^{2^{n_0-1}}} < 1 - \lambda < 1 = \frac{y_0}{y_0} = y_0 (\lambda x_0)^{2^{n_0}} \lambda^{n_0-1}.
    \end{equation*}
    Let $k \leq n_0-1$. Observe that $y_k = \lambda^k y_0$ and furthermore, if $x_0, \dots, x_k > 0$, then $x_k \leq x_0^{2^k}$. If one of $x_0, \dots, x_{n_0-1}$ is already nonpositive, then we are done. Thus, we may assume that all $x_0, \dots, x_{n_0-1}$ are positive.
    For $k \leq n_0-1$, we compute
    \begin{equation*}
        x_k y_k \leq x_0^{2^k} \lambda^k y_0 \leq x_0^{2^{n_0-1}} \lambda^{n_0-1} y_0 < 1 - \lambda
    \end{equation*}
    where we used that $x_0 \lambda > 1$. Thus, we have $1 - x_k y_k > \lambda$, and therefore
    \begin{equation*}
        x_{k+1} = x_k^2 (1 - x_k y_k) \geq x_k^2 \lambda.
    \end{equation*}
    Induction over $k$ yields
    \begin{equation*}
        x_{k+1} \geq x_0^{2^{k+1}} \lambda^{2^{k+1} - 1}
    \end{equation*}
    for all $k \leq n_0-1$. For $k = n_0 - 1$, we obtain
    \begin{equation*}
        x_{n_0}y_{n_0} \geq x_0^{2^{n_0}} \lambda^{2^{n_0} - 1} \lambda^{n_0} y_0 = y_0 (\lambda x_0)^{2^{n_0}} \lambda^{n_0-1} = 1
    \end{equation*}
    which implies $1 - x_{n_0}y_{n_0} \leq 0 $ and thus 
    \begin{equation*}
        x_{n_0+1} = x_{n_0}^2 (1 - x_{n_0}y_{n_0}) \leq 0.
    \end{equation*}
    Now that we have proven the existence of such $y_0$, the Intermediate Value Theorem implies the existence of $0 \leq \widetilde{y}_0 \leq y_0$ such that $\pi_z F^{n}(x_0,\widetilde{y}_0) = 0$. 
    Since this holds for arbitrarily small $y_0$ and since the basin of attraction at $0$ is bulging (as shown in \cite{RosayRudin1988} and \cite{Lilov_2004}), the orbits of $(x_0, \widetilde{y}_0)$ get caught in the bulged basin of attraction at $0$ for $y_0$ small enough, and thus stay bounded.
\end{example}


\section{Bulging Wandering Domains}\label{SECTION:3}

Having constructed examples where bulging fails, we now turn to the opposite question and investigate conditions under which Fatou components do bulge.

If the skew-product has an attracting invariant fiber where the one-dimensional dynamics is given by a function as in \Autoref{LEMMA:WDExamplesin1D}, the following theorem provides a sufficient condition on the perturbation for the wandering domains to bulge.

\begin{theorem}\label{THEOREM:BULGINGORDER}
    Consider the holomorphic skew-product $F\colon\CC^2\to\CC^2$
    \begin{equation}\label{eq:BULGINGORDER:Skew}
        F(z,w) = (f(z) + w h(z,w), g(w)) := (z + T + p(z) + w h(z,w), g(w))
    \end{equation}
    where $p$ is an entire $T$-periodic function for some $T \in \CC^{*}$, the function $z + p(z)$ has an attracting fixed point at $z_0 \in \CC$, $g$ is entire with an attracting fixed point at $0$, and $h$ is entire.  Assume that there exists $\delta > 0$ with
    \begin{equation}\label{THEOREM:BULGINGORDER:EQ1}
        \sum_{k = 0}^{\infty} \rho_k \max_{\substack{\lvert w \rvert \leq \rho_k, \\ \lvert z - (z_0 + kT) \rvert \leq \delta}} \lvert h(z,w) \rvert < \infty
    \end{equation}
    where $(\rho_k)_{k\in\NN}$ is a sequence of positive real numbers 
    such that $g(\overline{D(0, \rho_k)}) \subset \overline{D(0, \rho_{k+1})}$ holds for every $k \in \NN$. Then the wandering domain $U_0$ of $f$ at $z_0$ bulges to a Fatou component of $F$.
\end{theorem}


\begin{proof}
    Let $U_0$ be the wandering domain of $f$ around $z_0$, the properties of which are described in \Autoref{LEMMA:WDExamplesin1D}. 
    We show that $U_0$ bulges.


    Up to applying the global conjugation $(z,w) \mapsto (z + z_0, w)$, we may assume that $z_0 = 0$. 
    Thus, $0$ is an attracting fixed point of $ b(z) := z + p(z)$. 
    Let $ \delta > 0 $ be given by the hypotheses of the statement. 
    Up to shrinking it, we may assume that $\lvert b(z) \rvert < \lvert z \rvert$ for all $z \in D(0,\delta)$. 
    Due to the definition of $f$ and the $T$-periodicity of $p$, we obtain
    \begin{equation}\label{THEOREM:BULGINGORDER:PROOF:Attraction}
        \begin{split}
            \lvert f(z) - (k+1)T \rvert &= \lvert  z + p(z) + T - (k+1)T  \rvert = \lvert (z-kT) + p(z-kT) \rvert \\
            &= \lvert b(z-kT) \rvert < \lvert z - kT \rvert
        \end{split}
    \end{equation}
    for all $z \in D(kT, \delta)$ and $k \in\NN$. 

    Furthermore, since $0$ is an attracting fixed point of $g$, we may assume, by taking smaller values if necessary, that $\rho_k$ converges to $0$ as $k \to \infty$.

    The following lemma provides the key step of the proof. 
    \begin{lemma}\label{THEOREM:BULGINGORDER:PROOF:SUBLEMMA1}
        Under the hypotheses of \Autoref{THEOREM:BULGINGORDER} with $\lim_{k\to\infty}\rho_k=0$, there exists $L \in \NN$ such that for every $l \geq L$, as $k \to \infty$, both $z_k = \pi_z F^k(z,w)$ diverges uniformly to infinity and $w_k = \pi_w F^k(z,w)$ tends uniformly to $0$ on $ D(lT, \frac{\delta}{2}) \times D(0,\rho_l) $.
    \end{lemma}
    Depending on our choice of compactification, this lemma yields that $F^{n}(z,w)$ converges uniformly on $D\left( lT, \frac{\delta}{2} \right) \times D(0,\rho_l)$ to $\infty \in \Chat{2}$, $[1:0:0] \in\PC{2}$, or $(\infty,0) \in (\Chat{})^2$, respectively, as $n \to \infty$.
    Therefore, we have
    \begin{equation}\label{THEOREM:BULGINGORDER:EQ2}
        D\left( lT, \frac{\delta}{2} \right) \times D(0,\rho_l) \subset \F(F)
    \end{equation}
    for all $l \geq L$.
    To complete the proof, consider any $z \in U_0$. 
    Since $0$ is an attracting fixed point of the function $b(z) = z + p(z)$, there exists $N \geq L$ such that $b^N(z) \in D\left( 0,\frac{\delta}{2} \right)$.
    Consequently, we obtain for the orbit under $F$ that
    \begin{equation*}
        F^N(z,0) = (f^N(z), 0) = (b^N(z) + NT, 0) \in D\left( NT,\frac{\delta}{2} \right) \times \lbrace 0 \rbrace \subset \F(F)
    \end{equation*}
    where we used \eqref{THEOREM:BULGINGORDER:EQ2} in the last step.
    Using the backward invariance of the Fatou set, established in \Autoref{LEMMA:BackwardInvariance}, we deduce that $(z,0)$ belongs to the Fatou set of $F$, which completes the proof of Theorem~{\ref{THEOREM:BULGINGORDER}}.
\end{proof}

\begin{proof}[Proof of Lemma~{\ref{THEOREM:BULGINGORDER:PROOF:SUBLEMMA1}}]
    First, observe that the convergence of the series in \eqref{THEOREM:BULGINGORDER:EQ1} yields an integer $L \in \NN$ such that
    \begin{equation*}
        \sum_{k = L}^{\infty} \rho_k \max_{\substack{\lvert \xi \rvert \leq \rho_k, \\ \lvert \zeta - kT \rvert \leq \delta}} \lvert h(\zeta,\xi) \rvert < \frac{\delta}{2}.
    \end{equation*}
    Let $l\in\NN$ be such that $l \geq L$. 
    The uniform convergence of $w_k$ to $0$ in a small neighborhood of $0$ follows directly from the fact that $0$ is a attracting fixed point of $g$.
    By our assumption that $\rho_k$ tends to $0$ as $k \to \infty$, this holds in particular on $D(0,\rho_l)$.

    Inductively, we prove for all $(z,w) \in D\left( lT, \frac{\delta}{2} \right) \times D(0,\rho_l) $ and all $k \in \NN$ that
    \begin{equation}\label{THEOREM:BULGINGORDER:PROOF:NewInduction1}
        \left\lvert  z_k - (l+k)T  \right\rvert < \sum_{j=l}^{l + k-1} \rho_j \max_{\substack{\lvert \xi \rvert \leq \rho_j, \\ \lvert \zeta - jT \rvert \leq \delta}} \lvert h(\zeta,\xi) \rvert + \frac{\delta}{2}
    \end{equation}
    where $z_k = \pi_z F^k(z,w)$.
    Here, empty sums are interpreted as $0$.
    This already implies that $z_k $ diverges uniformly to $\infty$ on $ D\left( lT, \frac{\delta}{2} \right) \times D(0,\rho_l) $ as $k \to \infty$, since
    \begin{equation*}
        \begin{split}
            \sum_{j=l}^{l + k-1} \rho_j \max_{\substack{\lvert \xi \rvert \leq \rho_j, \\ \lvert \zeta - jT \rvert \leq \delta}} \lvert h(\zeta,\xi) \rvert + \frac{\delta}{2}            &\leq \sum_{j = L}^{\infty} \rho_j \max_{\substack{\lvert \xi \rvert \leq \rho_j, \\ \lvert \zeta - jT \rvert \leq \delta}} \lvert h(\zeta,\xi) \rvert  + \frac{\delta}{2} < \delta
        \end{split}
    \end{equation*}
    and $(l+k)T$ diverges to $\infty$ as $k \to \infty$.

    Let $z \in D\left( lT, \frac{\delta}{2} \right)$ and $w \in D(0,\rho_l)$. 
    The base case $k = 0$ of the induction is evident. 
    Assume that \eqref{THEOREM:BULGINGORDER:PROOF:NewInduction1} holds for some $k\in\NN$. 
    In particular, it follows that $z_k \in D((l+k)T, \delta)$.
    Furthermore, by our hypotheses on the sequence $(\rho_k)_{k\in\NN}$, we have $\lvert w_k \rvert = \lvert g^{k}(w) \rvert \leq \rho_{l+k}$. 
    As a consequence, we may apply \eqref{THEOREM:BULGINGORDER:PROOF:Attraction} and obtain
    \begin{equation*}
        \begin{split}
            | z_{k+1} - (l+k + 1)T |            &= \lvert  f(z_{k}) + w_{k} h(z_{k}, w_{k}) - (l+k+1)T \rvert \\
            &\leq \lvert  f(z_{k}) - (l+k+1)T \rvert + \lvert w_{k} h(z_{k}, w_{k}) \rvert \\
            &\leq \lvert  z_{k} - (l+k)T \rvert + \lvert w_{k} \rvert \lvert h(z_{k}, w_{k}) \rvert\\
            &\leq \lvert  z_{k} - (l+k)T \rvert + \rho_{l+k} \max_{\substack{\lvert \xi \rvert \leq \rho_{l + k}, \\ \lvert \zeta - (l+k)T \rvert \leq \delta}} \lvert h(\zeta,\xi) \rvert \\
            &< \sum_{j=l}^{l + k} \rho_{j} \max_{\substack{\lvert \xi \rvert \leq \rho_j, \\ \lvert \zeta - jT \rvert \leq \delta}} \lvert h(\zeta,\xi) \rvert + \frac{\delta}{2}\,.
        \end{split}
    \end{equation*}
    This completes the induction and therefore the proof of the lemma.
\end{proof}

A more readily verifiable, though not equivalent, condition for \eqref{THEOREM:BULGINGORDER:EQ1} can be given using the order of an entire function in one variable. We need the following definition.

\begin{definition}[cf. {\cite[Sections 1.1 and 1.2]{Ronkin1989}}]\label{DEF:Order}
    Let $f\colon \CC^2 \to \CC$ be entire and nonconstant. The \emph{maximum modulus} of $f$ at radii $r_1, r_2 >0$ is defined by
    \begin{equation*}
        M(r_1, r_2; f) := \max_{\lvert z \rvert \leq r_1, \lvert w \rvert \leq r_2} \lvert f(z, w) \rvert.
    \end{equation*}
    The \emph{order of $f$} is defined as
    \begin{equation*}
        \rho(f) := \limsup_{r\to\infty} \frac{\log^{+} \log M(r, r; f)}{\log r} \geq 0.
    \end{equation*}
    The \emph{order of $f$ in the $z$ variable} is defined as
    \begin{equation*}
        \rho_1(f) := \limsup_{r_1\to\infty} \frac{\log^{+} \log M(r_1, r_2; f)}{\log r} \geq 0
    \end{equation*}
    for some $r_2 > 0$.
\end{definition}
\begin{remark}
    As discussed in \cite[Section 1.2]{Ronkin1989}, the definition of $\rho_1(f)$ is independent of the choice of $r_2 > 0$.
\end{remark}

The following corollary gives a more practical criterion in terms of the growth of the perturbation.

\begin{corollary}\label{COROLLARY:GROWTH}
    Let $F\colon\CC^2\to\CC^2$ be a skew-product as in \eqref{eq:BULGINGORDER:Skew}.
    The wandering domain $U_0$ of $f$ at $z_0$ bulges to a Fatou component of $F$ if
    \begin{enumerate}
        \item $0$ is a geometrically attracting fixed point of $g$ and $\rho_1(h) < 1$, or
        \item $0$ is a superattracting fixed point of $g$ and $\rho_1(h) < \infty$.
    \end{enumerate}
\end{corollary}

\begin{proof}
    This corollary can be proven by itself using the same approach as \Autoref{THEOREM:BULGINGORDER}, but it also suffices to verify that the hypotheses of \Autoref{THEOREM:BULGINGORDER} hold. 
    As in the previous proof, we may assume that $z_0 = 0$. Let $\delta > 0$.

    Before looking at the two cases separately, note that $k\lvert T \rvert + \delta < C (k+1)$ with  $C := \lvert T \rvert + \delta$ for every $k \in \NN$. In particular, we have
    \begin{equation}\label{THEOREM:BULGINGORDER:PROOF:ConstantC}
        (k\lvert T \rvert + \delta)^{\beta} < C^{\beta} (k+1)^{\beta}
    \end{equation}
    for every $\beta > 0$ and $k \in \NN$.
    
    We start with the case $\rho_1(h) < 1$.
    Since $0$ is an attracting fixed point of $g$, there exist $0 < \alpha < 1$ and $\delta_g > 0$ such that $\lvert g(w) \rvert \leq \alpha \lvert w \rvert$ for all $w \in \overline{D(0,\delta_g)}$. 
    Observe, that this already implies that $w_n = g^n(w)$ converges to $0$ uniformly on $D(0,\delta_g)$ as $n\to\infty$. 
    In particular, choosing $\rho_k := \alpha^k \delta_g$ for $ k \in \NN $, we obtain $\lim_{k\to\infty}\rho_k = 0$ and $g(\overline{D(0, \rho_k)}) \subset \overline{D(0, \rho_{k+1})}$.

    Since $\rho_1(h) < 1$, we can choose $\tau > 0$ such that $\beta := \rho_1(h) + \tau < 1$. 
    By definition, there exists $L > 0$ such that for all $\sigma \geq L$ we have $\log^{+} \log M(\sigma, \delta_g; h) \leq \beta \log \sigma $ and so, $M(\sigma, \delta_g; h) \leq e^{\sigma^{\beta}}$.
    Since $0 < \beta < 1$ and $\log \alpha < 0$, there exists $K \in \N$ large enough so that for all $k \geq K$, we have $\log \alpha + C^{\beta} \frac{(k+1)^\beta}{k} < 0$
    which implies that
    \begin{equation*}
        q := e^{\log \alpha + C^{\beta} \frac{(K+1)^\beta}{K}} < 1.
    \end{equation*}
    Up to further increasing $K$, we may assume that $k\lvert T \rvert + \delta \geq L$ for all $k \geq K$.
    Together with \eqref{THEOREM:BULGINGORDER:PROOF:ConstantC}, this gives
    \begin{equation*}
        \begin{alignedat}{2}
            \sum_{k=K}^{\infty} \alpha^k e^{( k\lvert T \rvert + \delta )^{\beta}} &= \sum_{k=K}^{\infty} e^{k \log \alpha + ( k\lvert T \rvert + \delta )^{\beta}}&&\leq \sum_{k=K}^{\infty} e^{k \log \alpha + C^{\beta}(k+1)^\beta} \\
            &\leq \sum_{k=K}^{\infty} e^{k \left( \log \alpha + C^{\beta}\frac{(k+1)^\beta}{k} \right)} &&\leq \sum_{k=K}^{\infty} q^k < \infty.
        \end{alignedat}
    \end{equation*}
    Hence we obtain
    \begin{equation*}
        \begin{split}
            \sum_{k = 0}^{\infty} \rho_k \max_{\substack{\lvert w \rvert \leq \rho_k, \\ \lvert z - kT \rvert \leq \delta}} \lvert h(z,w) \rvert &\leq \sum_{k = 0}^{\infty} \delta_g \alpha^k M(k\lvert T \rvert + \delta, \delta_g; h) \\
            &\leq  \delta_g \sum_{k = 0}^{K-1} \alpha^k M(k\lvert T \rvert + \delta, \delta_g; h) + \delta_g\sum_{k = K}^{\infty} \alpha^k e^{( k\lvert T \rvert + \delta )^{\beta}} < \infty
        \end{split}
    \end{equation*}
    completing the proof when $\rho_1(h) < 1$.

    We now turn to the second case, that is, we assume that $0$ is super-attracting for $g$ and $\rho_1(h) < \infty$. 
    Hence, we have $g'(0) = \dots = g^{(d-1)}(0) = 0$ and $g^{(d)}(0) \neq 0$ for some $d \geq 2$. 
    As a result, there exists $\widetilde{g}$ entire such that $g(w) = w^d \widetilde{g}(w)$ for all $w \in \CC$. 
    For $\delta_g > 0$, define
    \begin{equation*}
        C_g := \max_{\lvert w \rvert \leq \delta_g} \lvert \widetilde{g}(w) \rvert > 0.
    \end{equation*}
    Up to shrinking $\delta_g$, we may assume that $\delta_g^d C_g \leq \delta_g$ and $\lvert g(w) \rvert \leq \lvert w \rvert$ for all $w \in \overline{D(0,\delta_g)}$.
    This ensures that for all $w \in D(0,\delta_g)$, we have
    \begin{equation}\label{THEOREM:BULGINGORDER:PROOF:WIterates}
        \lvert w_k \rvert = \lvert g^k(w) \rvert \leq C_g^{1 + d + \dots + d^{k-1}} \lvert w \rvert^{d^k} = C_g^{\frac{1 - d^k}{1 - d}} \lvert w \rvert^{d^k}.
    \end{equation}
    Set
    \begin{equation*}
        D := \frac{1}{1-d} \log C_g.
    \end{equation*}
    As before, but this time for $\tau = 1$ and $\beta := \rho_1(h) + 1$, there exists $L > 0$ such that for all $\sigma \geq L$, we have $\log^{+} \log M(\sigma, \delta_g; h) \leq \beta \log \sigma $, and so $M(\sigma, \delta_g; h) \leq e^{\sigma^{\beta}}$.    Given any $0 < q < 1$, we can choose $K \in \N$ large enough and $t > 0$ small enough such that for all $k \geq K$, we have $\log t - D + C^{\beta} \frac{(k+1)^\beta}{d^k} < \log q < 0$ and hence,
    \begin{equation*}
        e^{\log t - D + C^{\beta} \frac{(k+1)^\beta}{d^k}} < q < 1
    \end{equation*}
    for all $k \geq K$.
    Up to further increasing $K$, we may assume that $k\lvert T \rvert + \delta \geq L$ for all $k \geq K$.
    Using \eqref{THEOREM:BULGINGORDER:PROOF:ConstantC}, we obtain
    \begin{equation}
        \label{eq:BULGINGORDER:COR:e1}
        \begin{alignedat}{2}
            \sum_{k=K}^{\infty} e^{d^k (\log t - D)} e^{( k\lvert T \rvert + \delta )^{\beta}} &= \sum_{k=K}^{\infty} e^{d^k (\log t - D) + ( k\lvert T \rvert + \delta )^{\beta}} &&\leq \sum_{k=K}^{\infty} e^{d^k (\log t - D) + C^{\beta}(k+1)^\beta} \\
            &\leq \sum_{k=K}^{\infty} e^{d^k \left( \log t - D + C^{\beta}\frac{(k+1)^\beta}{d^k} \right)} &&\leq \sum_{k=K}^{\infty} q^{d^k} < \infty.
        \end{alignedat}
    \end{equation}
    Up to shrinking $t$, we may assume that $t < \delta_g$. Setting $\rho_0 := t $ and $\rho_k := C_g \rho_{k-1}^d$, we have $g(\overline{D(0, \rho_k)}) \subset \overline{D(0, \rho_{k+1})}$ for all $k \in \NN$.
    By induction, it further follows that for all $k \in \NN$, we have
    \begin{equation*}
        \rho_k  = C_g^{\frac{1 - d^k}{1 - d}} t^{d^k} = e^{d^k (\log t - D) + D}.
    \end{equation*}
    In particular, we obtain $\lim_{k\to\infty} \rho_k $ = 0.
    Using the explicit formula for $\rho_k$ together with \eqref{THEOREM:BULGINGORDER:PROOF:ConstantC} and \eqref{eq:BULGINGORDER:COR:e1}, and arguing similarly to the first case, we obtain that
    \begin{equation*}
        \begin{split}
            \sum_{k = 0}^{\infty} \rho_k \max_{\substack{\lvert w \rvert \leq \rho_k, \\ \lvert z - kT \rvert \leq \delta}} \lvert h(z,w) \rvert &\leq \sum_{k = 0}^{\infty} \rho_k M(k\lvert T \rvert + \delta, \delta_g; h) \\
            &\leq  \sum_{k = 0}^{K-1} \rho_k M(k\lvert T \rvert + \delta, \delta_g; h) + \sum_{k = K}^{\infty} C_g^{\frac{1 - d^k}{1 - d}} t^{d^k} e^{( k\lvert T \rvert + \delta )^{\beta}} \\
            &\leq  \sum_{k = 0}^{K-1} \rho_k M(k\lvert T \rvert + \delta, \delta_g; h) + C_g^{\frac{1}{1-d}} \sum_{k = K}^{\infty} e^{d^k (\log t - D)} e^{( k\lvert T \rvert + \delta )^{\beta}} < \infty,
        \end{split}
    \end{equation*}
    completing the proof for the second case.
\end{proof}

	\bibliographystyle{amsalpha}
	\bibliography{bibliography_wd}
\end{document}